\documentclass[preprint, 11pt, english, compress]{elsarticle}
\usepackage{amsmath}
\usepackage{latexsym, amssymb}
\usepackage{amsthm}
\usepackage{txfonts,xcolor}
\usepackage[normalem]{ulem}

\newtheorem{thm}{Theorem}[section] 

\newtheorem{cor}[thm]{Corollary}

\newtheorem{lem}[thm]{Lemma}
\newtheorem{prop}[thm]{Proposition}

\theoremstyle{definition}
\newtheorem{rem}[thm]{Remark}
\newtheorem{exmpl}[thm]{Example}
\newtheorem{defn}[thm]{Definition}
\newtheorem{algo}[thm]{Algorithm}

\DeclareMathOperator{\ch}{char}
\DeclareMathOperator{\sgn}{sgn}
\DeclareMathOperator{\sn}{sn}

\renewcommand{\Re}{\mathop{\mathfrak{Re}}\nolimits}
\renewcommand{\Im}{\mathop{\mathfrak{Im}}\nolimits}

\newcommand\operA[2]{{\if!#2!\operatorname{#1}\else{\operatorname{#1}_{#2}^{\phantom{I}}}\fi}} 

\newcommand\Cref[1]{{Corollary~\ref{#1}}}%

\def\tr{{\operatorname{Tr}}}

\def\norm{{\operatorname{Norm}}}

\newcommand{\Trace}[1][]{\if!#1!\operatorname{Tr}\else{\operatorname{Tr}_{#1}^{\phantom{I}}}\fi} 

\long\def\forget#1\forgotten{{}} %

\def\({\left(}
\def\){\right)}

\newif\iffurther
\furtherfalse

\newif\ifXY 
\XYtrue     
\ifXY

\input xy
\input xyidioms.tex
\usepackage{xy}
\xyoption{all} %
\fi 

\usepackage{babel}

\makeatletter
\def\ps@pprintTitle{%
	\let\@oddhead\@empty
	\let\@evenhead\@empty
	\def\@oddfoot{}%
	\let\@evenfoot\@oddfoot}
\makeatother

\begin{document}

\begin{frontmatter}

\title{Roots and Critical Points of Polynomials over Cayley--Dickson Algebras}
\author{Adam Chapman}
\ead{adam1chapman@yahoo.com}
\address{School of Computer Science, Academic College of Tel-Aviv-Yaffo, Rabenu Yeruham St., P.O.B 8401 Yaffo, 6818211, Israel}

\author{Alexander Guterman}
\ead{guterman@list.ru}
\address{ 
Department of Mathematics and Mechanics, \\
Lomonosov  Moscow State University, Moscow, 119991, Russia \\ 
Moscow Center for Fundamental and Applied Mathematics, Moscow, 119991, Russia\\
Moscow Center for Continuous Mathematical Education,  Moscow, 119002, Russia
}

\author{Solomon Vishkautsan}
\ead{wishcow@gmail.com}
\address{Department of Computer Science, Tel-Hai Academic College, Upper Galilee, 12208 Israel}

\author{Svetlana Zhilina}
\ead{s.a.zhilina@gmail.com}
\address{ 
Department of Mathematics and Mechanics, \\
Lomonosov Moscow State University, Moscow, 119991, Russia \\
Moscow Center for Fundamental and Applied Mathematics, Moscow, 119991, Russia\\
Moscow Center for Continuous Mathematical Education,  Moscow, 119002, Russia
}

\begin{abstract}
We study the roots of polynomials over Cayley--Dickson algebras over an arbitrary field and of arbitrary dimension.
For this purpose we generalize the concept of spherical roots from quaternion and octonion polynomials to this setting, and demonstrate their basic properties. We show that the spherical roots (but not all roots) of a polynomial $f(x)$ are also roots of its companion polynomial $C_f(x)$ (defined to be the norm of $f(x)$). For locally-complex Cayley--Dickson algebras, we show that the spherical roots of $f'(x)$ (defined formally) belong to the convex hull of the roots of $C_f(x)$, and we also prove that all roots of $f'(x)$ are contained in the snail of $f(x)$, as defined by Ghiloni and Perotti for quaternions. The latter two results generalize the classical Gauss--Lucas theorem to the locally-complex Cayley--Dickson algebras, and we also generalize Jensen's classical theorem on real polynomials to this setting.
\end{abstract}


\end{frontmatter}

\section{Introduction}

Cayley--Dickson algebras $\{ A_n \}_{n \in \mathbb{N}}$ over an arbitrary field $F$ are a family of $2^n$-dimensional algebras over $F$ with involution, which are obtained from a two-dimensional algebra $A_1$ (or, if $\ch(F) \neq 2$, from $A_0 = F$) by repeating the Cayley--Dickson doubling process an arbitrary number of times. Each step of this process is determined by a nonzero parameter $\gamma_k \in F$ such that $A_{k+1} = A_k \{ \gamma_k \}$ (see Section~\ref{sec:prelim} for more details).

For example, starting with $A_0=\mathbb{R}$ and taking $\gamma_k = -1,$ for an integer $k\ge{0}$, at each step of the process, we obtain a sequence of $2^n$-dimensional $\mathbb{R}$-algebras which we call the real algebras of the main sequence.
In particular, we obtain $A_1=\mathbb{C}$ and $A_2=\mathbb{H}$, Hamilton's real quaternion algebra; the next algebra in the sequence is the real octonion division algebra $\mathbb{O}$, and the next one is the real sedenion algebra $\mathbb{S}$. Another important example of real Cayley--Dickson algebras is the algebra of the split-quaternions $\hat{\mathbb{H}}$, constructed by choosing $A_0=\mathbb{R}, \gamma_0=-1$ and $\gamma_1=1$, which is isomorphic to the algebra of real $(2 \times 2)$-matrices $M_2(\mathbb{R})$, see~\cite[p.~157]{McCrimmon:2004}.

The algebraic structure gets looser the higher we go in the Cayley--Dickson sequence: the complex numbers ($A_1$) are not ordered, the quaternion algebras ($A_2$) are not commutative but still associative, the octonion algebras ($A_3$) are not associative but still alternative, and the rest are not alternative. The lack of structure of $A_n$ for $n\geq 4$ makes it difficult to obtain profound results,
but the motivation to study these algebras remains \cite{biss2008large, bremner2001identities, brevsar2011locally, chan2006conjugacy, EakinSathaye:1990, moreno2006alternative}; there are also potential applications of  Cayley--Dickson algebras in physics  (e.g., \cite{furey2018three, gillard2019three, kuwata2004born}). 


The goal of this paper is to improve our understanding of the behavior of the roots of polynomials over Cayley--Dickson algebras. We also present results for the special case of polynomials over locally-complex Cayley--Dickson algebras: a real unital algebra is called locally-complex if any nonscalar element generates a subalgebra isomorphic to $\mathbb{C}$ (see \cite{brevsar2011locally} for an in-depth study of locally-complex algebras; in Section~\ref{sec:prelim} we show that locally-complex Cayley--Dickson algebras are exactly the real algebras of the main sequence). The roots of quaternion and octonion polynomials over an arbitrary field $F$ are well-understood (see, e.g., \cite{Chapman:2020b,Chapman:2020a,ChapmanMachen:2017}).
The roots of higher-dimensional Cayley--Dickson algebras are more difficult to handle, but we prove in Section~\ref{sec:roots}
that the ``spherical roots'' of a polynomial $f(x)$ can be recovered from the companion polynomial (defined as $C_f(x)=\overline{f(x)}f(x)$), and we provide an explicit algorithm for finding them.

The classical Gauss--Lucas theorem states that the critical points of a polynomial $f(x)$ with complex coefficients lie within the convex hull defined by the roots of $f(x)$ (see \cite[\S{6}]{Marden:1966}).
One can ask whether the Gauss--Lucas theorem extends to other locally-complex Cayley--Dickson algebras. In \cite{Ghiloni:2018} three versions of the theorem were considered for the quaternion algebra $\mathbb{H}$:
\begin{enumerate}[(1)]
    \item the roots of $f'(x)$ lie in the convex hull of the roots of $f(x)$,
    \item the roots of $f'(x)$ lie in the convex hull of the roots of $C_f(x)$, the companion polynomial of $f(x)$,
    \item the roots of $f'(x)$ lie in the ``snail'' (see Section~\ref{sec:gauss--lucas}) of $f(x)$.
\end{enumerate}

Ghiloni and Perotti showed in \cite{Ghiloni:2018} that (1) is in general false for polynomials of degree at least~2, (2) is false for degree~3 and above, but holds for quadratic polynomials, and (3) holds true in general.
In Section~\ref{sec:gauss--lucas} we prove that (2) holds true for the spherical roots of $f(x)$ over a locally-complex Cayley--Dickson algebra, and for arbitrary roots of quadratic octonion polynomials, and that (3) holds true for polynomials over any locally-complex Cayley--Dickson algebra.  At the end of this section we prove some analogues of the classical bounds for the roots and critical points of complex polynomials via their coefficients.

Finally, in Section~\ref{sec:jensen} we generalize Jensen's classical theorem (see \cite[\S{7}]{Marden:1966}) to locally-complex Cayley--Dickson algebras, improving the Gauss--Lucas bound for the special case of polynomials with real coefficients.
\section{Preliminaries} \label{sec:prelim}

\subsection{Cayley--Dickson algebras}

In this section we briefly discuss the important properties and notations used in the theory of Cayley--Dickson algebras that will be needed in the article. For more in-depth background on Cayley--Dickson algebras see, for example, the renowned paper \cite{Schafer:1954}, the classical monographs \cite{McCrimmon:2004,SSSZ}, and references therein. We start with the inductive definition of Cayley--Dickson algebras over an arbitrary field $F$.

Given an algebra $A$ with an involution $a \mapsto \bar{a}$ over a field $F$ and an element $\gamma \in F^\times = F \setminus \{ 0 \}$, the Cayley--Dickson doubling process produces a new algebra $B = A \{ \gamma \}$ which is defined as the set of ordered pairs of elements of $A$, i.e, $A \times A$, with the following operations. Multiplication by elements from $F$ and addition are componentwise, and multiplication by elements from $B$ is given by $(a,b)(c,d) = (ac + \gamma \bar{d}b, da + b\bar{c})$ for any $a,b,c,d\in A$. The involution then extends to $B$ by $\overline{(a,b)} = (\bar{a},-b)$.

To obtain what we call Cayley--Dickson algebras, we first take a quadratic \'{e}tale extension $A_1=K$ of $F$ with the nontrivial Galois automorphism $a \mapsto \bar{a}$. By~\cite[p.~32, Theorem~1]{SSSZ}, $K$ is isomorphic to the algebra $F[\ell_1 : \ell_1^2=\ell_1+\mu] = F + F \ell_1$ for some $\mu \in F$ with $4 \mu + 1 \neq 0$, and $\overline{\alpha + \beta \ell_1} = (\alpha + \beta) - \beta \ell_1$ for all $\alpha, \beta \in F$. Then applying the Cayley--Dickson process to $A_1$ with an element  $\gamma_1\in{F^\times}$ gives rise to a quaternion algebra $A_2 = A_1 \{ \gamma_1 \}$; applying it again with  $\gamma_2\in{F^\times}$ gives rise to an octonion algebra $A_3 = A_2 \{ \gamma_2 \}$, 
and so on. By repeating this process an arbitrary number of times, one obtains an infinite sequence $\{A_n\}_{n\in\mathbb{N}}$ of algebras called Cayley--Dickson algebras.
We remark that for $\ch(F)\ne{2}$ the algebra $A_1$ is isomorphic to $F\{ \gamma_0\}$, where $\gamma_0 = (4\mu + 1)/4 \neq 0$, so one can start the Cayley--Dickson process from $A_0=F$, with $a \mapsto \bar{a}$ being the identity map.

We recall a convenient notation for Cayley--Dickson algebras which extends the common notation used for quaternions and octonions over an arbitrary field
(see \cite[Pages 30--31]{SSSZ}). We write $A_n = [\mu, \gamma_1,\ldots,\gamma_{n-1})_F$ if $A_n = A_{n-1}\{\gamma_{n-1}\}$, $A_{n-1} = A_{n-2}\{\gamma_{n-2}\}$, and so on, until $A_1 = F[\ell_1 : \ell_1^2=\ell_1+\mu]$. If $\ch(F) \neq 2$, then $A_1 = F\{ \gamma_0\}$, so we can also write $A_n = (\gamma_0,\ldots,\gamma_{n-1})_F$. For example, $\mathbb{H}=(-1,-1)_{\mathbb{R}}$ and $\hat{\mathbb{H}} = (-1,1)_{\mathbb{R}}$. 
(See \cite{Schafer:1954}
for an alternative notation.)

All Cayley--Dickson algebras are quadratic, i.e., every nonscalar element $\lambda \in A_n$ together with the unit element $1$ generates a two-dimensional subalgebra containing $F$. More precisely, it satisfies $\lambda^2-\tr(\lambda)\lambda+\norm(\lambda)=0$, where the trace $\tr:A_n \rightarrow F$ is a linear map defined by $\lambda \mapsto \lambda + \overline{\lambda}$ and the norm $\norm: A_n \rightarrow F$ is a quadratic form defined by $\lambda \mapsto \overline{\lambda} \cdot \lambda$. The polynomial $p_{\lambda}(x) = x^2-\tr(\lambda)x+\norm(\lambda)$ is called the characteristic polynomial of $\lambda$. The trace and the norm can be computed inductively by using the following identities, see~\cite[p.~32, Theorem~1]{SSSZ} and~\cite[p.~435]{Schafer:1954}. Given $\alpha + \beta \ell_1 \in A_1$, we have
\begin{align*}
    \tr(\alpha + \beta \ell_1) &= 2 \alpha + \beta,\\
    \norm(\alpha + \beta \ell_1) &= \alpha^2 + \alpha \beta - \mu \beta^2,
\end{align*}
and for an arbitrary $(a,b) \in A_{n+1}$, $n \in \mathbb{N}$, we have
\begin{align*}
	\tr((a,b)) &= \tr(a),\\
	\norm((a,b)) &= \norm(a) - \gamma_{n} \norm(b).
\end{align*}
In the case of the real algebras of the main sequence, $\norm(\lambda)$ is the Euclidean norm, and thus it is anisotropic, i.e, $\norm(x) \neq 0$ for any $x \neq 0$, cf.~\cite[p.~5]{Moreno:1998}.


 \begin{prop}
 	A real Cayley--Dickson algebra $A$ is locally-complex if and only if $A$ is an algebra of the main sequence, i.e., $\gamma_0=\gamma_1=\ldots=-1$.
 \end{prop}
 
 \begin{proof}
 	The implication from right to left follows from~\cite[Example~4.4]{brevsar2011locally}. We now prove the implication from left to right. 
 	
 	 By~\cite[Exercise 2.5.1]{McCrimmon:2004}, any real Cayley--Dickson algebra $A_n = (\gamma_0, \dots, \gamma_{n-1})_{\mathbb{R}}$ is isomorphic to $A_n'= (\sgn(\gamma_0), \dots, \sgn(\gamma_{n-1}))_{\mathbb{R}}$.
 	 Hence it is sufficient to consider only $\gamma_k \in \{ \pm 1 \}$, $k = 0, \dots, n-1$.
 	
 	Assume to the contrary that there exists $k \in \{ 0, \dots, n - 1\}$ such that $\gamma_k = 1$.
 	Then the norm form on $A$ is not positive definite, so $A$ is not locally-complex by~\cite[Lemma~4.1(iv)]{brevsar2011locally}.
 \end{proof}
 
 For a locally-complex Cayley--Dickson algebra $A$ we define the real part of an element $\lambda\in A$ by $\Re(\lambda) = \frac{1}{2}\tr(\lambda)$, and the imaginary part of $\lambda$ is defined to be $\Im(\lambda)=\lambda - \Re(\lambda)$. The absolute value of $\lambda\in{A}$ is defined by $|\lambda|=\sqrt{\norm(\lambda)}$.

\subsection{Alternative elements in Cayley--Dickson algebras}

Cayley--Dickson algebras $A_n$ are power-associative, that is, the subalgebra generated by any element is associative, and flexible, i.e.,  satisfy $(xy)x = x(yx)$ for all $x,y \in A_n$, see~\cite[p.~436]{Schafer:1954}. Moreover, up to octonions, the algebras $A_n$ are composition algebras, i.e., the norm form is multiplicative, but the rest of them are not, see~\cite[p.~164, Theorem~2.6.1]{McCrimmon:2004}. In addition, up to octonion algebras, the algebras $A_n$ are division if and only if the norm form is anisotropic,
but the rest are never division algebras (i.e., they contain zero divisors), even when the norm form is anisotropic.

We define the standard basis $L_n = \{ \ell_m^{(n)} \; | \; m = 0 , \dots, 2^n - 1\}$ of $A_n$ inductively:
\begin{enumerate}[(1)]
    \item $\ell_0^{(1)} = 1$ and $\ell_1^{(1)} = \ell_1$;
    \item $\ell_m^{(n+1)} = \begin{cases}
    (\ell_m^{(n)}, 0), & 0 \leq m \leq 2^n - 1,\\
    (0, \ell_{m - 2^n}^{(n)}), & 2^n \leq m \leq 2^{n+1} - 1,
    \end{cases}
    \quad n \in \mathbb{N}$.
\end{enumerate}
If $\ch(F) \neq 2$, then there is also another standard basis $E_n = \{ e_m^{(n)} \; | \; m = 0 , \dots, 2^n - 1\}$ of~$A_n$:
\begin{enumerate}[(1)]
    \item $e_0^{(0)} = 1$;
    \item $e_m^{(n+1)} = \begin{cases}
    (e_m^{(n)}, 0), & 0 \leq m \leq 2^n - 1,\\
    (0, e_{m - 2^n}^{(n)}), & 2^n \leq m \leq 2^{n+1} - 1,
    \end{cases}
    \quad n \in \mathbb{N}_{0}$.
\end{enumerate}
We will often omit upper indices when they are clear from the context. It can be easily seen that $\ell_0^{(n)} = e_0^{(n)}$ is the unit element of $A_n$, so we will also denote $\ell_0^{(n)} = e_0^{(n)} = 1$. The elements of $E_n \setminus \{ 1 \}$ anticommute pairwise, that is, $e_k^{(n)} e_m^{(n)} = - e_m^{(n)} e_k^{(n)}$ for $1 \leq k < m \leq 2^n - 1$, hence it is more natural to use $E_n$ rather than $L_n$ if possible. Therefore, we will use the basis $L_n$ whenever $F$ has arbitrary characteristic, whereas $E_n$ will be used mostly for real Cayley--Dickson algebras.

An element $a \in A_n$ is called alternative if $a(ab) = a^2 b$ and $(ba)a = b a^2$ for all $b \in A_n$. The next proposition implies that all elements of $E_n$ and $L_n$ are alternative (see also~\cite[Lemma~4]{Schafer:1954} for a different proof of alternativity of $E_n$).

\begin{prop} \label{prop:inherit-alternativity}
Let $a,b \in A_n$ be alternative. Assume also that either at least one of them belongs to $F$, or $F + Fa = F + Fb$. Then $(a,b)$ is alternative in $A_{n+1}$.
\end{prop}

\begin{proof}
The condition that $a \in F$, $b \in F$, or $F + Fa = F + Fb$ guarantees that $a(cb) = (ac)b$ for all $c \in A_n$. Then the statement follows by direct computation, see also~\cite[Lemma~4.4]{biss2008large}  and~\cite[Theorem~3.3]{moreno2006alternative} where the case of the real algebras of the main sequence was considered.
\end{proof}

\begin{cor} \label{cor:alternative-basis}
All elements of $E_n$ and $L_n$ are alternative in $A_n$. Moreover, any element of the form $\alpha e_{2k}^{(n)} + \beta e_{2k+1}^{(n)}$ or $\alpha \ell_{2k}^{(n)} + \beta \ell_{2k+1}^{(n)}$, where $0 \leq k \leq 2^{n-1} - 1$ and $\alpha, \beta \in F$, is alternative in $A_n$.
\end{cor}

\begin{proof}
It is clear that for all $n \geq 2$ we have
$$
\alpha \ell_{2k}^{(n)} + \beta \ell_{2k+1}^{(n)} = 
\begin{cases}
(\alpha \ell_{2k}^{(n-1)} + \beta \ell_{2k+1}^{(n-1)}, 0), & 0 \leq k \leq 2^{n-2} - 1,\\
(0, \alpha \ell_{2(k - 2^{n-2})}^{(n-1)} + \beta \ell_{2(k - 2^{n-2})+1}^{(n-1)}), & 2^{n-2} \leq k \leq 2^{n-1} - 1,
\end{cases}
$$
and a similar expression is valid for $\alpha e_{2k}^{(n)} + \beta e_{2k+1}^{(n)}$. Hence the statement follows from Proposition~\ref{prop:inherit-alternativity} by induction on $n$.
\end{proof}

\begin{rem}
Proposition~\ref{prop:inherit-alternativity} provides a method to obtain many other alternative elements in an arbitrary Cayley--Dickson algebra. In particular, any element of the form $\alpha e_k^{(n)} + \beta e_{k + 2^{n-1}}^{(n)} = (\alpha e_k^{(n-1)}, \beta e_k^{(n-1)})$ or $\alpha \ell_k^{(n)} + \beta \ell_{k + 2^{n-1}}^{(n)} = (\alpha \ell_k^{(n-1)}, \beta \ell_k^{(n-1)})$, where $0 \leq k \leq 2^{n-1} - 1$ and $\alpha, \beta \in F$, is also alternative in $A_n$.

However, we are specially interested in the elements of the form $\alpha \ell_{2k}^{(n)} + \beta \ell_{2k+1}^{(n)}$, since for any $\lambda = \sum_{m=0}^{2^n-1} \lambda_m \ell_m^{(n)}$ we have $\norm(\lambda) = \sum_{k=0}^{2^{n-1}-1} \norm(\lambda_{2k} \ell_{2k}^{(n)} + \lambda_{2k+1} \ell_{2k+1}^{(n)})$, and $\norm(\lambda_{2k} \ell_{2k}^{(n)} + \lambda_{2k+1} \ell_{2k+1}^{(n)})$ can be easily computed from $\norm(\lambda_{2k} + \lambda_{2k+1} \ell_1)$ in $A_1$, see Lemma~\ref{lem:if-spherical} below.
\end{rem}

\subsection{Cayley--Dickson polynomials}

Let $A$ be a Cayley--Dickson algebra over a field $F$. The ring of polynomials $A[x]$ over $A$ is defined to be $A \otimes_F F[x]$. Note that we have the chain of rings $A \subset A[x] \subset A \otimes_F F(x)$, where $A \otimes_F F(x)$ is
the Cayley--Dickson algebra over the function field $F(x)$ in one indeterminate over $F$. The element $x$ is therefore central in $A[x]$, so it commutes and alternates with all other elements. Therefore, every polynomial $f(x)$ in $A[x]$ can be written as $f(x)=a_n x^n+\dots+a_1 x+a_0$, by placing the coefficient on the left-hand side of the indeterminate in each monomial. The involution $a \mapsto \bar{a}$ extends to $A[x]$ by acting trivially on $x$, that is, $\overline{f(x)} = \overline{a}_n x^n + \dots + \overline{a}_1 x + \overline{a}_0$.
We define the substitution of an element $\lambda$ from $A$ in $f(x)$ by $f(\lambda)=a_n (\lambda^n)+\dots+a_1 \lambda+a_0$. The substitution is well-defined because $A$ is power-associative. Note, however, that we cannot expect in general for $f(x)=g(x)h(x)$ to imply $f(\lambda)=g(\lambda)h(\lambda)$ for $\lambda\in{A}$.

We define the companion polynomial $C_f(x)$ of $f(x)$ to be $\norm(f(x))=\overline{f(x)}\cdot f(x)$. The companion polynomial $C_f(x)$ is in $F[x]$, i.e., has central coefficients. 

\begin{lem} {\rm \cite[Theorem~3.3]{Chapman:2020a}, \cite[Theorem~3.6]{ChapmanMachen:2017}} \label{lem:companion-octonion}
	Let $A$ be a quaternion or octonion algebra, and let $f(x)\in{A[x]}$, then any root of $f(x)$ is also a root of the companion polynomial $C_f(x)$.
\end{lem}
The latter lemma is false for higher-dimensional Cayley--Dickson algebras as the following example shows:
\begin{exmpl}\label{exmpl:companion-fails}
	Consider $A=\mathbb{S}$, the real sedenion algebra with the standard basis $\{ e_m \; | \; m = 0, \dots, 15\}$. Set $a = e_1 + e_{10}$ and $b = e_7 + e_{12}$.
	Then $ab=0$ (see~\cite{Moreno:1998}).
	The polynomial $f(x)=ax$ has $b$ as a root, whereas the only root of $C_f(x)=\norm(a)x=2x^2$ is zero.
Lemma~\ref{lem:companion-octonion} does not hold for monic polynomials either. Indeed, one can consider a monic polynomial $g(x) = x^2 + ax + 2$, and then $C_g(x) = x^4 + 6 x^2 + 4$. Since $b^2 = -2$, we have $g(b) = 0$ but $C_g(b) \neq 0$.
\end{exmpl}

\section{Spherical roots of Cayley--Dickson polynomials} \label{sec:roots}

Let $A$ be a Cayley--Dickson algebra over a field $F$.
A root $\lambda\in{A\setminus F}$ of $f(x)\in {A[x]}$ is called a ``spherical root'' of $f(x)$ if all the elements $r \in A$ satisfying
$p_{\lambda}(r) = 0$ (where $p_\lambda(x) = x^2-\tr(\lambda)x+\norm(\lambda)$ is the characteristic polynomial of $\lambda$),  are also roots of $f(x)$. In the special case of quaternion and octonion division algebras, a root $\lambda\notin{F}$ is spherical if and only if all elements in the conjugacy class of $\lambda$ are roots of $f(x)$ (see \cite{Chapman:2020a}).

\begin{prop} \label{prop:roots-from-field}
Let $A$ be a Cayley--Dickson algebra such that the norm form $\norm(a)$ is anisotropic. If $\lambda\in F$,
then $\lambda$ is the unique root of the quadratic polynomial $p_{\lambda}(x)$.
\end{prop}

\begin{proof}
Since $\lambda \in F$, we have $p_{\lambda}(x) = x^2-2\lambda{x}+\lambda^2=(x-\lambda)^2$.
If $r$ is a root of $p_{\lambda}(x)$, then it is contained in a two-dimensional subalgebra $F[r]$ of $A$ which is commutative and associative. In this subalgebra $(r-\lambda)^2=0$ implies $\norm(r - \lambda) \cdot (r - \lambda) = \overline{(r - \lambda)}(r - \lambda)^2=0$. Since $\norm(r - \lambda) \neq 0$ whenever $r - \lambda \neq 0$, this is only possible for $r = \lambda$.
\end{proof}

\begin{exmpl} \label{exmpl:roots-from-field}
Proposition~\ref{prop:roots-from-field} does not hold for general Cayley--Dickson algebras, since it is possible that $(r - \lambda)^2 = 0$ for $r - \lambda \neq 0$. For example, one can consider the element $x = e_1 + e_2 \in \hat{\mathbb{H}}$ which satisfies $\tr(x) = \norm(x) = 0$, and thus $x^2 = 0$.
\end{exmpl}

In order to characterize spherical roots of polynomials over Cayley--Dickson algebras, we define an equivalence relation on the Cayley--Dickson algebra $A$. We say that $r\in A$ is quadratically-equivalent to $\lambda\in A$ if the quadratic polynomials 
$p_r(x)$ and $p_{\lambda}(x)$ coincide, that is, $\tr(r)=\tr(\lambda)$ and $\norm(r)=\norm(\lambda)$. Clearly, quadratic-equivalence is an equivalence relation indeed. Moreover, we see that $\lambda\in A \setminus F$ is a spherical root of $f\in{A[x]}$ if and only if the quadratic-equivalence class of $\lambda$
is contained in the set of roots of $f(x)$.

\begin{prop}
\label{prop:3.3}
If $A$ is a Cayley--Dickson algebra, $\lambda \in A$, and $r\in A\setminus{F}$,
then $r$ and $\lambda$ are quadratically-equivalent if and only if $p_{\lambda}(r) = 0$.
\end{prop}

\begin{proof}
The implication from left to right is straightforward. Assume now that $p_{\lambda}(r) = r^2 - \tr(\lambda) r + \norm(\lambda) = 0$. We also have $p_r(r) = r^2 - \tr(r) r + \norm(r) = 0$, so we obtain $(\tr(r) - \tr(\lambda))r - (\norm(r) - \norm(\lambda)) = 0$. Since $r \notin F$, it follows that $\tr(r)=\tr(\lambda)$ and $\norm(r)=\norm(\lambda)$, and thus $r$ and $\lambda$ are quadratically-equivalent.
\end{proof}

\begin{exmpl} \label{exmpl:quadratic-equivalence}
We cannot allow $r$ to be in $F$ for the following reason:
consider the algebra $\hat{\mathbb{H}} \cong M_2(\mathbb{R})$.
In this algebra, one can take $\lambda=\left(\begin{array}{lr}
t & 0\\
0 & s
\end{array}\right)$ and $r=\left(\begin{array}{lr}
t & 0\\
0 & t
\end{array}\right)$ for two different real numbers $t$ and $s$.
Then $p_{\lambda}(x) = (x-s)(x-t)$ is not equal to $p_r(x) = (x-t)^2$, however, $p_\lambda(r)=0$.
\end{exmpl}


When $A = A_n$ is a quaternion or octonion division algebra over an arbitrary field $F$ (i.e., $n \in \{ 2, 3\}$, and the norm on $A$ is anisotropic), the quadratic-equivalence class of $\lambda$ coincides with the conjugacy class of~$\lambda$, that is, the set of elements of the form $h \lambda h^{-1}$ with $h \in A_n$ being invertible, see~\cite[Remarks~3.1 and~5.3]{Chapman:2020a}. Clearly, this is not true for $n = 1$, since for any $a \in A_1$ the elements $a$ and $\bar{a}$ are always quadratically-equivalent, but they are not conjugate unless $a = \bar{a}$, due to commutativity and associativity of $A_1$. As the following example shows, this is not the case for $n \geq 4$ either, though the expression $h \lambda h^{-1}$ is well-defined for any $h \in A_n$ such that $\norm(h) \neq 0$, since $h^{-1} = (\norm(h))^{-1} \bar{h}$ and $A_n$ is flexible.

\begin{exmpl}
Let $A$ be at least $16$-dimensional real locally-complex Cayley--Dickson algebra. Consider $a = e_1 + e_{10}$ and $b = e_7 + e_{12}$ from Example~\ref{exmpl:companion-fails}. Then $bab^{-1} = 0$, but $\norm(a) = 2 \neq 0 = \norm(0)$, so $a$ and $0$ are not quadratically-equivalent.
\end{exmpl}

For locally-complex Cayley--Dickson algebras we can say more about spherical roots. We begin with the following characterizations of quadratic-equivalence, which are generalizations of well known facts about conjugacy classes in real quaternions (see, e.g., the survey article \cite[Theorem 2.2]{zhang:1997}) and octonions (\cite[Theorem 3.4]{Chapman:2020a}). 

\begin{prop} \label{prop:quadratic-equivalence-LCCDA}
	Let $A$ be a locally-complex Cayley--Dickson algebra, and $\lambda\in{A}$. Then $r$ and $\lambda$ are quadratically-equivalent if and only if $\Re(r)=\Re(\lambda)$ and $|\Im(r)|=|\Im(\lambda)|$.
\end{prop}

\begin{proof} It is easy to verify that
	$\Re(r)=\Re(\lambda)$ and $|\Im(r)|=|\Im(\lambda)|$ if and only if $\tr(r)=\tr(\lambda)$ and $\norm(r)=\norm(\lambda)$, and the latter means quadratic-equivalence.
\end{proof}

\begin{cor}
	Let $A$ be a locally-complex Cayley--Dickson algebra, and $\lambda\in{A\setminus F}$. Denote $a = \Re(\lambda)$ and $b=|\Im(\lambda)|$. Let $i\in{A}\setminus{F}$ be a certain element such that $i^2=-1$, so that $i$ generates a complex subfield $C$ of $A$. Then $r=a\pm bi$ are quadratically-equivalent to $\lambda$, and they are the two roots of 
	$p_{\lambda}(x)$ in~$C$.
\end{cor}


The following is a reformulation of the previous corollary.

\begin{cor}
	Let $A$ be a locally-complex Cayley--Dickson algebra, and $f\in{A[x]}$. If $\lambda\in{A\setminus F}$ is a spherical root of $f(x)$, then $f(x)$ has exactly two distinct roots quadratically-equivalent to $\lambda$ in every complex subfield of $A$. 
\end{cor}

\begin{exmpl}
 If $A$ is a locally-complex Cayley--Dickson algebra of dimension 4 or 8, i.e., the real quaternion or octonion algebra, and there exist two distinct conjugate roots of a polynomial, then these roots must be spherical roots (see \cite[Theorem 3.4]{Chapman:2020a}). This is false in general for locally-complex Cayley--Dickson algebras. Returning to Example~\ref{exmpl:companion-fails},  we see that
  both $\bar{b} = -b$ and $a$ are quadratically equivalent to $b$, and for $f(x) = ax$ we have $f(b) = f(-b) = 0$ but $f(a) = a^2 = -2 \neq 0$.
\end{exmpl}

\begin{lem}\label{lem:if-decompose}
Let $f(x)\in{A[x]}$ and $\lambda\in{A\setminus F}$, where $A$ is an arbitrary Cayley--Dickson algebra. If 
$f(x) = g(x) p_{\lambda}(x)$ for some $g(x)\in A[x]$, then $\lambda$ is a spherical root of~$f(x)$.
\end{lem}

\begin{proof}
	Denote $T=\tr(\lambda)$ and $N=\norm(\lambda)$ (recall that the elements $T$ and $N$ are central in $A$, i.e., they belong to $F$). Now
suppose that $r$ is a root of $p_{\lambda}(x) = x^2-Tx+N$. 
We denote $g(x)=a_nx^n+\ldots+ a_1x + a_0$, where $a_k\in{A}$ for $1\le{k}\le{n}$. For convenience, we also set $a_{-2}=a_{-1}=a_{n+1}=a_{n+2}=0$. Expanding the expression for $f(x)$, we get:
\begin{align*}
	f(x) &= (a_nx^n+\ldots+ a_1x + a_0)(x^2-Tx+N) \\
	&= a_nx^{n+2}+(-a_nT+a_{n-1})x^{n+1}+(a_nN-a_{n-1}T+a_{n-2})x^n \\
	& + \ldots + (a_kN-a_{k-1}T+a_{k-2})x^k + \ldots + (a_2N-a_{1}T+a_{0})x^2 \\
	& + (a_1N-a_0T)x+a_0N \\
	&= \sum_{k=0}^{n+2} (a_kN-a_{k-1}T+a_{k-2})x^k.
\end{align*}
Now, we substitute $r$ in $f(x)$ and get
\begin{align*}
	f(r) &= \sum_{k=0}^{n+2} (a_kN-a_{k-1}T+a_{k-2})r^k \\
	&= \sum_{k=0}^{n} (a_kr^{k+2}-a_kTr^{k+1}+a_{k}Nr^k) \\
	&= \sum_{k=0}^{n} a_k(r^k(r^2-Tr+N)) = 0.
\end{align*}
Therefore, $\lambda$ is a spherical root of $f(x)$. 
\end{proof}

\begin{lem} \label{lem:if-spherical}
Let A be an arbitrary Cayley--Dickson algebra, and $f(x) \in A[x]$. If $\lambda \in A \setminus F$ is a spherical root of $f(x)$, then
$f(x) = g(x) p_{\lambda}(x)$ for some $g(x) \in A[x]$.
\end{lem}

\begin{proof}
	As before, denote $T = \tr(\lambda)$ and $N = \norm(\lambda)$. We perform a division with remainder of $f(x)$ by the characteristic polynomial of $\lambda$, $p_{\lambda}(x) = x^2 - Tx + N$; i.e., there exist $g(x) \in A[x]$ and $a, b \in A$ such that
	$$
	f(x) = g(x)(x^2 - Tx + N) + ax + b.
	$$
	Denote $h(x) = g(x)(x^2 - Tx + N)$. 
We consider the decomposition $\lambda = \sum_{m=0}^{2^n-1} \lambda_m \ell_m$, where $L_n = \{ \ell_m \; | \; m = 0, \dots, 2^n - 1\}$ is the standard basis for $A = [\mu, \gamma_1,\ldots,\gamma_{n-1})_F$ and $\lambda_m \in F$ for all $m = 0, \dots, 2^n - 1$. Given $k = \sum_{l = 1}^{n-1} c_l 2^{l-1}$ with $c_l \in \{ 0, 1\}$ being the coefficients in the binary decomposition of $k$, we define $\nu_k = \prod_{l = 1}^{n-1} \gamma_l^{c_l} \in F$. Clearly, $\nu_k \neq 0$ for any $k \in \{ 0, \dots, 2^{n-1} - 1 \}$, and $\nu_0 = 1$. Then we have $\tr(\lambda) = 2 \lambda_0 + \lambda_1$ and
	$$
	\norm(\lambda) = \sum_{k = 0}^{2^{n-1}-1} \nu_k \norm(\lambda_{2k} + \lambda_{2k+1}\ell_1) = \sum_{k = 0}^{2^{n-1}-1} \nu_k (\lambda_{2k}^2 + \lambda_{2k} \lambda_{2k+1} - \mu \lambda_{2k+1}^2).
	$$
	
	Since $\lambda \notin F$, there exists some $m \neq 0$ such that $\lambda_m \neq 0$. Assume first that $m = 2k+1$ for some $k \geq 0$. Then we consider $\kappa = \lambda + \lambda_{2k+1} (\ell_{2k} - 2 \ell_{2k+1})$. Note that $(\lambda_{2k} + \lambda_{2k+1}) - \lambda_{2k+1} \ell_1 = \overline{\lambda_{2k} + \lambda_{2k+1} \ell_1}$ in $A_1$, so their traces and norms coincide, and thus $\tr(\lambda) = \tr(\kappa) = T$ and $\norm(\lambda) = \norm(\kappa) = N$. Since $\lambda$ is a spherical root of $f(x)$, we have $f(\lambda) = f(\kappa) = 0$. By Lemma~\ref{lem:if-decompose}, $\lambda$ and $\kappa$ are also roots of $h(x)$, so we obtain $a \lambda + b = a \kappa + b = 0$. It follows that $\lambda_{2k+1} a(\ell_{2k} - 2 \ell_{2k+1}) = a(\kappa - \lambda) = 0$. By Corollary~\ref{cor:alternative-basis}, $\ell_{2k} - 2 \ell_{2k+1}$ is alternative in~$A$, so
	\begin{align*}
	    0 &= \lambda_{2k+1} \big(a(\ell_{2k} - 2 \ell_{2k+1})\big) \overline{(\ell_{2k} - 2 \ell_{2k+1})}\\
	    &= \lambda_{2k+1} a\left((\ell_{2k} - 2 \ell_{2k+1}) \overline{(\ell_{2k} - 2 \ell_{2k+1})}\right)\\
	    &= \norm(\ell_{2k} - 2 \ell_{2k+1}) \lambda_{2k+1} a = -(4\mu + 1) \nu_k\lambda_{2k+1} a
	\end{align*}
	Then $(4\mu + 1) \nu_k \lambda_{2k+1} \neq 0$ implies that $a = 0$, and therefore, $b = 0$. Finally, $f(x) = g(x) (x^2 - Tx + N)$ as required.
	
	Assume now that $\lambda_m = 0$ for all odd values of $m$. Since $\lambda \notin F$, this is only possible for $n \geq 2$. It is sufficient to show that there exists $r = \sum_{m=0}^{2^n-1} r_m \ell_m \in A$ which is quadratically-equivalent to $\lambda$, and $r_m \neq 0$ for some odd value of $m$. 
	Let $k \geq 1$ be such that $\lambda_{2k} \neq 0$. Consider two cases:
	\begin{itemize}
	    \item If $\mu \neq 0$ then we set $r_m = \lambda_m$ for all $m \neq 2k+1$. We also set $r_{2k+1} = \mu^{-1} \lambda_{2k} \neq 0$. Then $\norm(r_{2k} + r_{2k+1} \ell_1) = r_{2k}^2 + r_{2k} r_{2k+1} - \mu r_{2k+1}^2 = \lambda_{2k}^2 = \lambda_{2k}^2 + \lambda_{2k} \lambda_{2k+1} - \mu \lambda_{2k+1}^2 = \norm(\lambda_{2k} + \lambda_{2k+1} \ell_1)$. Hence $\tr(r) = \tr(\lambda)$ and $\norm(r) = \norm(\lambda)$, as required.
	    \item If $\mu = 0$ then we choose $r_0 \in F$ such that $r_0^2 \neq \lambda_0^2$ and set $r_1 = 2(\lambda_0 - r_0)$. We also set $r_m = \lambda_m$ for all $m \notin \{ 0, 1, 2k+1 \}$. We next find $r_{2k+1}$ from the linear equation $(r_0^2 + r_0 r_1) + \nu_k \lambda_{2k} r_{2k+1} = \lambda_0^2$. Then at least one of the elements $r_1$ and $r_{2k+1}$ is nonzero.
	    
	    We now have $\tr(r) = 2r_0 + r_1 = 2 \lambda_0 = 2 \lambda_0 + \lambda_1 = \tr(\lambda)$ and $(r_0^2 + r_0 r_1) + \nu_k (r_{2k}^2 + r_{2k} r_{2k+1}) = \lambda_0^2 + \nu_k \lambda_{2k}^2 = (\lambda_0^2 + \lambda_0 \lambda_1) + \nu_k (\lambda_{2k}^2 + \lambda_{2k} \lambda_{2k+1})$, so $\norm(r) = \norm(\lambda)$. \qedhere
	\end{itemize}
\end{proof}

Combining Lemmas~\ref{lem:if-decompose} and~\ref{lem:if-spherical}, we obtain the following theorem.

\begin{thm} \label{thm:decompose-iff-spherical}
Let $A$ be an arbitrary Cayley--Dickson algebra, and $f(x)\in{A[x]}$. Then $\lambda\in A \setminus F$ is a spherical root of~$f(x)$ if and only if
$f(x) = g(x) p_{\lambda}(x)$ for some $g(x)\in A[x]$.
\end{thm}

\begin{thm} \label{thm:spherical-root-of-companion}
Let $A$ be a Cayley--Dickson algebra over a field $F$
and $f(x) \in A[x]$. Assume that  $\lambda\in{A \setminus F}$ is a spherical root of $f(x)$. Then $\lambda$ is a root of~$C_f(x)$.
\end{thm}

\begin{proof}
Since $\lambda$ is a spherical root of $f(x)$, by Theorem~\ref{thm:decompose-iff-spherical} we have
$f(x)=g(x) p_{\lambda}(x)$ for some $g(x)\in A[x]$. Now, $C_f(x)=\overline{f(x)}\cdot f(x)=(p_{\lambda}(x) \overline{g(x)})\cdot (g(x)p_{\lambda}(x))$. Since the trace and the norm belong to $F$, the polynomial $p_{\lambda}(x)$ is central in $A[x]$. Hence $C_f(x)=\overline{g(x)}\cdot g(x) \cdot (p_{\lambda}(x))^2$.
Therefore, $\lambda$ is a root of $C_f(x)$.
\end{proof}

\begin{cor} \label{cor:number-of-spherical-roots}
Let A be a locally-complex Cayley--Dickson algebra, and $f(x) \in A[x]$.
If the degree of $f(x)$ is $n$, then it has at most $\lfloor \frac{n}{2} \rfloor$ 
different classes of spherical roots with respect to quadratic equivalence.
\end{cor}

\begin{proof}
By Theorem~\ref{thm:decompose-iff-spherical}, $p_\lambda(x)$ is a factor of $f(x)$. Since $p_\lambda(x)$ is central in $A[x]$, it is also a factor of $\overline{f(x)}$, and thus $(p_\lambda(x))^2$ is a factor of $C_f(x)$.
Since different indecomposable quadratic real polynomials are co-prime as polynomials in $\mathbb{C}[x]$, having more than $\lfloor \frac{n}{2} \rfloor$ spherical roots would imply that the degree of $C_f(x)$ be at least $4 \cdot (\lfloor \frac{n}{2} \rfloor+1)$. However, the degree of $C_f(x)$ is $2n$, contradiction.
\end{proof}

\begin{exmpl}
Corollary~\ref{cor:number-of-spherical-roots} does not hold for general Cayley--Dickson algebras. For example, let $A = \hat{\mathbb{H}}$ and $f(x) = (x-1)x(x+1)$. Then $f(x) = (x - 1) (x^2 + x) = x (x^2 - 1) = (x + 1) (x^2 - x)$, and there exist $\lambda_1, \lambda_2, \lambda_3 \in \hat{\mathbb{H}} \setminus \mathbb{R}$ such that $p_{\lambda_1}(x) = x^2 + x$, $p_{\lambda_2}(x) = x^2 - 1$, and $p_{\lambda_3}(x) = x^2 - x$, see Example~\ref{exmpl:quadratic-equivalence}. Hence there are three quadratic classes of spherical roots for $f(x)$.
\end{exmpl}

\begin{algo}
One can therefore find the spherical roots of a polynomial $f(x)\in A[x]$ over an arbitrary Cayley--Dickson algebra:
\begin{enumerate}
\item Find the roots of the companion polynomial $C_f(x)$ in $A$, which are grouped into a finite number of classes (since there is a finite number of monic quadratic polynomials in $F[x]$ which divide $C_f(x)$), where each class is characterized by its trace and norm.
\item For any such root $\lambda$, reduce $f(x)$ to a linear expression in $x$ by the rule $x^2=\tr(\lambda)x-\norm(\lambda)$ (in fact, this is equivalent to division with remainder of $f(x)$ by $p_{\lambda}(x)$).
\item If this reduction ends with 0, then $\lambda$ is a spherical root of $f(x)$. Otherwise, it is not.
\end{enumerate}
\end{algo}

In what follows, we would like to push the idea of Theorem~\ref{thm:decompose-iff-spherical} a bit further, and to prove that we can extend the decomposition of $f(x)$ given there, for the special case of Cayley--Dickson division algebras.
We start with preliminary lemmas.

\begin{lem}
	Let $A$ be an alternative algebra over a field $F$, and let $f(x) \in A[x]$, $g(x)\in F[x]$ and $\lambda\in{A}$. Denote $h(x)=f(x)g(x)$. Then $h(\lambda)=f(\lambda)g(\lambda)$. 
\end{lem}

	Notice that under the assumptions of the lemma, $g(\lambda)$ belongs to $F[\lambda]$, and therefore belongs to the associative subalgebra $F[\lambda,c]$ of $A$ for any $c \in A$; thus the proof proceeds almost the same as for associative algebras \cite[Lemma 4.3]{ChapmanVishkautsan:2021}.

\begin{proof}
	Write $f(x)=\sum_{i=0}^m c_i x^i$ and $g(x)=\sum_{j=0}^k d_j x^j$.
	Then $h(x)=f(x)\cdot g(x)=\sum_{i=0}^m (c_i x^i)\cdot g(x)=\sum_{i=0}^m \sum_{j=0}^k (c_i d_j) x^{i+j}$.
	Hence $h(\lambda)=\sum_{i=0}^m \sum_{j=0}^k (c_i d_j) \lambda^{i+j}.$ Since the coefficients $d_i$, $0\le{i}\le{k}$, are in $F$, we get $h(\lambda)=\sum_{j=0}^k d_j \sum_{i=0}^m  (c_i \lambda^{i+j})= \sum_{j=0}^k d_j ((\sum_{i=0}^m  c_i \lambda^{i})\lambda^j)=\sum_{j=0}^k d_j (f(\lambda)\lambda^j)=f(\lambda)\sum_{j=0}^k d_j \lambda^j=f(\lambda)g(\lambda)$.
\end{proof}

\begin{cor} \label{cor:root-of-product}
	Let $A$ be a division algebra over a field $F$, and let $f(x) \in A[x]$, $g(x)\in F[x]$ and $\lambda\in{A}$. Denote $h(x)=f(x)g(x)$. If $h(\lambda)=0$ and $g(\lambda)\ne 0$, then $f(\lambda)=0$.
\end{cor}

\begin{thm}
	Let $A$ be an octonion or quaternion division algebra over an arbitrary field F,  and $f(x) \in A[x]$. Then $f(x)$ decomposes 
	into a product 
	\[f(x)=g(x)\prod_{i=1}^k p_{\lambda_i}(x),\]
	where $g(x)$ has only non-spherical roots, which are exactly the non-spherical roots of $f(x)$, and $\lambda_1,\ldots,\lambda_k$ are representatives from all the quadratic-equivalence classes of spherical roots of $f$, counting multiplicity.
\end{thm}

\begin{proof}
The proof is by induction on $k$, the number of spherical roots of $f$, using Theorem~\ref{thm:decompose-iff-spherical} and Corollary~\ref{cor:root-of-product}.
\end{proof}

\begin{cor}
	Let A be an octonion or quaternion division algebra over an arbitrary field F, and $f(x) \in A[x]$.
	Suppose $f$ is of even degree $2n$, and that $f(x)$ has $n$ spherical roots from $n$ distinct quadratic-equivalence classes. Then $f(x)=cg(x)$ for some $c\in A$ and $g(x)\in F[x]$. 
\end{cor}

\section{Derivatives and critical points} \label{sec:gauss--lucas}

Let $A$ be an arbitrary algebra over a field $F$. The formal derivative $f'(x)$ of a polynomial $f(x)=a_n x^n+\dots+a_1 x+a_0 \in A[x]$ is defined as $f'(x)=n a_n x^{n-1}+\dots+a_1$. 

\begin{lem}[Leibniz rule]
For $f(x),g(x)\in{A[x]}$, if $h(x)=f(x)g(x)$ then $h'(x)=f'(x)g(x)+f(x)g'(x)$.
\end{lem}

\begin{proof}
The proof of this lemma is straightforward,  
since the formal derivative is linear, and standard proofs of the Leibniz rule for polynomial rings do not rely on the associativity of the ring of coefficients. 
\end{proof}

An element $\lambda \in A$ is called a critical point of $f(x)\in{A[x]}$ if $f'(\lambda)=0$. A spherical critical point is a critical point which is also a spherical root of $f'(x)$. 

\begin{thm} \label{thm:gauss--lucas}
Let $f(x)\in{A[x]}$ and $\lambda\in{A}$, where $A$ is a locally-complex Cayley--Dickson algebra.
The spherical critical points of $f(x)$ are contained in the convex hull of the roots of $C_f(x)$.
\end{thm}

\begin{proof}
Since $\lambda$ is a spherical critical point of $f(x)$, we have
$f'(x)=g(x)p_{\lambda}(x)$ by Theorem~\ref{thm:decompose-iff-spherical}.
Note that $\left(\overline{f(x)}\right)' = \overline{f'(x)}$, hence
\begin{align*}
 C_f'(x)&=\overline{f'(x)}f(x)+\overline{f(x)}f'(x)\\
 &=\overline{g(x)p_{\lambda}(x)} \cdot f(x) + \overline{f(x)}g(x)p_{\lambda}(x) \\
 &=(\overline{g(x)}f(x)+\overline{f(x)}g(x))p_{\lambda}(x).
\end{align*}

\noindent Therefore, $\lambda$ is a root of $C_f'(x)$. Besides, $\lambda$ belongs to the isomorphic copy $C$ of the field of complex numbers over $\mathbb{R}$ containing $\lambda$. Thus, by the Gauss--Lucas theorem it lies in the convex hull (inside this field of complex numbers) of the roots of $C_f(x)$ from that field. It is clear that this convex hull is contained in the convex hull defined by all the roots in $A$ of $C_f(x)$.
\end{proof}

\begin{thm} \label{thm:quadratic-polynomial}
Let $\mathbb{O}$ be the real octonion algebra and $f(x) \in \mathbb{O}[x]$ be a quadratic polynomial. Then the roots of $f'(x)$ lie in the convex hull of the roots of $C_f(x)$.
\end{thm}

\begin{proof}
By \cite{Chapman:2020b}, $f(x)$ factors into linear pieces $f(x)=(c(x-a))(x-b)=cx^2-(ca+cb)x+(ca)b$.
Then $f'(x)=2cx-(ca+cb)$.
The only root of $f'(x)$ is thus $\frac{1}{2}(a+b)$.
The element $b$ is a root of $f(x)$, and hence a root of $C_f(x)$ as well. The element $a$ is not necessarily a root of $f(x)$, but it is a root of $C_f(x)$, and therefore the only root of $f'(x)$ is in the convex hull of the roots of $C_f(x)$.
\end{proof}

If the degree of $f(x)$ is at least $3$, then Theorem~\ref{thm:quadratic-polynomial} does not hold even in the case of $\mathbb{H}$, see~\cite[Corollary~1]{Ghiloni:2018}. Besides, as the following example shows, if $A$ is at least $16$-dimensional locally-complex Cayley--Dickson algebra, then Theorem~\ref{thm:quadratic-polynomial} does not hold also for quadratic polynomials over $A$.

\begin{exmpl}
Let $A$ be at least $16$-dimensional locally-complex Cayley--Dickson algebra. Consider $a = e_1 + e_{10}$ and $b = e_7 + e_{12}$ from Example~\ref{exmpl:companion-fails}, and let $f(x) = ax^2$. Then $f'(x) = 2ax$ and $C_f(x) = 2x^4$, so any element from $\mathbb{R}b$ is a root of $f'(x)$, but the convex hull of the roots of $C_f(x)$ is $\{ 0 \}$.
\end{exmpl}

%
%
%

Let $\langle a, b \rangle$ denote a real-valued symmetric bilinear form associated with the quadratic form $\norm(a)$ on an arbitrary Cayley--Dickson algebra $A$ over $\mathbb{R}$. Then $\langle a, a \rangle = \norm(a)$ for all $a \in A$. If $A$ is a locally-complex Cayley--Dickson algebra over $\mathbb{R}$, then $\langle a, b \rangle$ is the Euclidean inner product, since $\norm(a)$ is the Euclidean norm squared.

\begin{lem} {\rm \cite[Lemma 1.3]{Moreno:1998}, \cite[Lemma 4.2]{Guterman-Zhilina:2020}} \label{lemma:inner-product-movement}
	Let $A$ be an arbitrary Cayley--Dickson algebra over $\mathbb{R}$, $a, b, c \in A$. Then $\langle a, bc \rangle = \langle a \bar{c}, b \rangle = \langle \bar{b}a, c \rangle$.
\end{lem}

Let $A$ be a locally-complex Cayley--Dickson algebra.
Then any element $I \in A$ of trace $0$ and norm $1$ generates a subalgebra $\mathbb{R}[I] \cong \mathbb{C}$ via $I \mapsto i$, which we denote by $\mathbb{C}_I$. We also denote by $\pi_I: A \to \mathbb{C}_I$ the orthogonal projection onto $\mathbb{C}_I$. Now any $f(x) \in A[x]$ decomposes as $f(x) = f_I(x) + f_I^{\perp}(x)$, where $f_I(x)$ has coefficients in $\mathbb{C}_I$ and $f_I^{\perp}(x)$ has coefficients in $\mathbb{C}_I^{\perp}$. Here $\mathbb{C}_I^{\perp} = \{ a \in A \; | \; \langle a, b \rangle = 0 \text{ for all } b \in \mathbb{C}_I \}$.

\begin{prop} \label{prop:projection-roots}
	Let $A$ be a locally-complex Cayley--Dickson algebra, $I \in A$ be an element of trace $0$ and norm $1$, $f(x) \in A[x]$. Then for any $\lambda \in \mathbb{C}_I$ we have $f_I(\lambda) \in \mathbb{C}_I$ and $f_I^{\perp}(\lambda) \in \mathbb{C}_I^{\perp}$. Thus $f(\lambda) = 0$ if and only if $f_I(\lambda) = f_I^{\perp}(\lambda) = 0$. 
\end{prop}

\begin{proof}
	Since $\mathbb{C}_I$ is a subalgebra in $A$, it is clear that $f_I(\lambda) \in \mathbb{C}_I$. It remains to prove that $f_I^{\perp}(\lambda) \in \mathbb{C}_I^{\perp}$, that is, $\langle f_I^{\perp}(\lambda), b \rangle = 0$ for any $b \in \mathbb{C}_I$. By linearity, it is sufficient to show that $\langle a_k \lambda^k, b \rangle = 0$, where $k \in \mathbb{N}_0$ and $a_k \in \mathbb{C}_I^{\perp}$. Indeed, by Lemma~\ref{lemma:inner-product-movement}, we have $\langle a_k \lambda^k, b \rangle = \langle a_k, b \overline{\lambda^k} \rangle = 0$, since $b \overline{\lambda^k} \in \mathbb{C}_I$. Then the statement follows.
\end{proof}

If $f_I(x)$ is not constant, we denote by $K_{\mathbb{C}_I}(f_I)$ the convex hull of the roots of $f_I(x)$ in $\mathbb{C}_I$. Otherwise, we set $K_{\mathbb{C}_I}(f_I) = \mathbb{C}_I$. The following definition is motivated by~\cite[Definition~2]{Ghiloni:2018}, where polynomials over $\mathbb{H}$ were considered.

\begin{defn}
Given a locally-complex Cayley--Dickson algebra $A$ and $f(x) \in A[x]$, the Gauss--Lucas snail $\sn(f)$ is defined to be the union of $K_{\mathbb{C}_I}(f_I)$, where $I$ ranges over all elements of trace $0$ and norm $1$ in~$A$.
\end{defn}

\begin{thm} \label{thm:Gauss-Lucas}
Let $A$ be a locally-complex Cayley--Dickson algebra, $f(x) \in A[x]$. Then the roots of $f'(x)$ are contained in $\sn(f)$.
\end{thm}

\begin{proof}
Consider a root $\lambda$ of $f'(x)$.
This element is contained in some $\mathbb{C}_I$.
Therefore, it is also a root of $f'_I(x)$.
An easy computation shows that $f'_I(x)$ is the derivative of $f_I(x)$.
Hence, by the (complex) Gauss--Lucas theorem, $\lambda$ is contained in the convex hull of the roots of $f_I(x)$, and therefore in $\sn(f)$.
\end{proof}

Since the snail is admittedly difficult to compute, one should look for a more practical bound for the critical points. Such bounds can be obtained from spherical bounds on the roots themselves. Given a polynomial $f(x) \in A[x]$, we denote by $\rho(f)$ the spectral radius of $f(x)$, that is,
$$
\rho(f) = \sup_{\substack{\lambda \in A, \\ f(\lambda) = 0}} |\lambda|.
$$

\begin{lem} \label{lem:cauchy-bound}
Let $A$ be a locally-complex Cayley--Dickson algebra, and let $f(x) = x^n + a_{n-1} x^{n-1} + \dots + a_1 x + a_0 \in A[x]$ be a monic polynomial with $\deg f \geq 1$. For any $r \in \sn(f)$ we have
\begin{align*}
    |r| &< R_1(f) = \sqrt{ 1 + |a_{n-1}|^2 + \dots + |a_1|^2 + |a_0|^2},\\
    |r| &< R_2(f) = 1 + \max_{0 \leq k \leq n-1} |a_k|,\\
    |r| &\leq R_3(f) = \max \left\{ 1, |a_{n-1}| + \dots + |a_1| + |a_0| \right\}.
\end{align*}
\end{lem}

\begin{proof}
Since $r \in \sn(f)$, there exists some $I$ such that $r \in K_{\mathbb{C}_I}(f_I)$.  The polynomial $f(x)$ is monic, so $f_I(x) = x^n + \pi_I(a_{n-1}) x^{n-1} + \dots + \pi_I(a_1) x + \pi_I(a_0)$ and $\deg f_I = \deg f \geq 1$. Hence $K_{\mathbb{C}_I}(f_I)$ is the convex hull of the roots of $f_I(x)$ in $\mathbb{C}_I$. It follows that $|r|$ is bounded from above by the maximum value of the moduli of the roots of $f_I(x)$. It is a classical result that the desired estimates hold for the roots of a complex polynomial $f_I(x)$, and its proof can be found in, e.g., \cite[Theorem~8.1.7(i)]{RahmanSchmeisser:2002} with $\lambda = 1$ and $p = 2$. It remains to note that $R_j$ are increasing functions of $|a_k|$, $j \in \{ 1, 2, 3 \}$, $k \in \{ 0, 1, \dots, n-1 \}$, and that, as is well known from linear algebra, $|\pi_I(a_k)| \le |a_k|$ for all~$k$.
\end{proof}

\begin{thm} \label{thm:compact-snail}
Let $A$ be a locally-complex Cayley--Dickson algebra, and let $f(x) \in A[x]$ be a monic polynomial with $\deg f \geq 1$. Then $\sn(f)$ is a compact subset of $A$ which is contained in an open ($j = 1, 2$) or a closed ($j = 3$) ball with center at the origin and radius $R_j(f)$.
\end{thm}

\begin{proof}
It is proved similarly to~\cite[Remarks~2 and~3]{Ghiloni:2018} that for any $I$ the set $K_{\mathbb{C}_I}$ is a compact subset of $\mathbb{C}_I$ which depends continuously on $I$, and that $\sn(f)$ is a closed subset of $A$. It remains to apply Lemma~\ref{lem:cauchy-bound} which shows that $\sn(f)$ is a bounded, and thus compact, subset of $A$.
\end{proof}

\begin{exmpl}
If $f(x) \in A[x]$ is a non-monic polynomial then $\sn(f)$ need be neither closed nor bounded even for $A = \mathbb{H}$, see~\cite[Example~1]{Ghiloni:2018}. The same example applies for an arbitrary locally-complex Cayley--Dickson algebra of dimension at least~4. Let $E_n = \{ e_m \; | \; m = 0, \dots, 2^n - 1\}$ be the standard basis of $A$, $i = e_1$. Consider $f(x) = ix^2 + x$. Then for an arbitrary $I \in A$ with trace~$0$ and norm~$1$ we have $f_I(x) = \alpha I x^2 + x$, where $\alpha = \langle I, i \rangle$. Hence $K_{\mathbb{C}_I}(f_I) = \{ 0 \}$ for $\alpha = 0$, and $K_{\mathbb{C}_I}(f_I)$ is a line segment from $0$ to $\alpha^{-1} I$ for $\alpha \neq 0$. It follows that $\sn(f) = \{ \lambda \in A \; | \; \tr(\lambda) = 0, \; 0 < \langle \lambda, i \rangle \leq 1 \} \cup \{ 0 \}$. Clearly, this set is neither closed nor bounded.
\end{exmpl}

\begin{cor}
Let $A$ be a locally-complex Cayley--Dickson algebra, and let $f(x) \in A[x]$ be a monic polynomial with $\deg f \geq 1$. Then the roots and the critical points of $f$ are contained in an open ($j = 1, 2$) or a closed ($j = 3$) ball with center at the origin and radius $R_j(f)$. In other words, 
$$
\rho(f), \: \rho(f') < R_1(f), \qquad \rho(f), \: \rho(f') < R_2(f), \qquad \rho(f), \: \rho(f') \leq R_3(f).
$$
\end{cor}

\begin{proof}
Since any root $r$ of $f(x)$ is contained in some $\mathbb{C}_I$, it follows from Proposition~\ref{prop:projection-roots} that it is also a root of $f_I(x)$, and thus $r \in \sn(f)$. Hence the estimates on $\rho(f)$ are obtained from Lemma~\ref{lem:cauchy-bound}. The estimates on $\rho(f')$ follow immediately from Theorem~\ref{thm:Gauss-Lucas} and Lemma~\ref{lem:cauchy-bound}.
\end{proof}

\begin{rem}
Other classical bounds on the roots of a complex polynomial which are given in~\cite[Theorem~8.1.7(i-ii)]{RahmanSchmeisser:2002} also hold for an arbitrary monic polynomial $f(x) \in A[x]$ and are proved similarly.
\end{rem}

If the dimension of $A$ is at least $4$, then it is not true in general that $\rho(f') \leq \rho(f)$. We construct a counterexample similarly to~\cite[Corollary~1]{Ghiloni:2018}.

\begin{exmpl}
Let $A = \mathbb{H}$ and $f(x) = (x-i)(x-j)(x-k) = x^3 - (i+j+k)x^2 + (i-j+k)x + 1$. Then $C_f(x) = (x^2+1)^3$. If $r$ is an arbitrary root of $f(x)$, then is also a root of $C_f(x)$, so $r^2 + 1 = 0$, and thus $|r| = 1$. Hence $\rho(f) = 1$.

We have $f'(x) = 3x^2 - 2(i+j+k)x + (i-j+k)$, so $C_{f'}(x) = 9x^4 + 12x^2 - 4x + 3$. It can be seen that all roots of $C_{f'}(x)$ in $\mathbb{C}$ have modulus greater than $1$, and thus the same holds for roots of $C_{f'}(x)$ in $\mathbb{H}$. Hence $\rho(f') > 1 = \rho(f)$.
\end{exmpl}

\begin{cor} \label{cor:nonmonic-polynomial}
For every polynomial $f(x)=a_n x^n+\dots+a_1 x+a_0 \in \mathbb{O}[x]$ with $n \geq 1$ and $a_n \neq 0$,
it holds that
\begin{align*}
    \rho(f), \: \rho(f') &< \widetilde{R}_1(f) = \tfrac{1}{|a_n|} \cdot \sqrt{ |a_n|^2 + |a_{n-1}|^2 + \dots + |a_1|^2 + |a_0|^2},\\
    \rho(f), \: \rho(f') &< \widetilde{R}_2(f) = 1 + \tfrac{1}{|a_n|} \cdot\max_{0 \leq k \leq n-1} |a_k|,\\
    \rho(f), \: \rho(f') &\leq \widetilde{R}_3(f) = \max \left\{ 1, \tfrac{1}{|a_n|} \cdot(|a_{n-1}| + \dots + |a_1| + |a_0|) \right\}.
\end{align*}
\end{cor}

\begin{proof}
Let $g(x) = a_n^{-1} f(x) = x^n+a_n^{-1}a_{n-1}x^{n-1}+\dots+a_n^{-1}a_1x+a_n^{-1}a_0$. Alternativity of the Cayley--Dickson algebra $A \otimes_F F(x)$ implies that $C_g(x) = \norm(a_n^{-1} f(x)) = \norm(a_n^{-1}) \norm(f(x)) = \norm(a_n^{-1}) C_f(x)$. Hence the roots of $C_g(x)$ and $C_f(x)$ coincide and, similarly, the roots of $C_{g'}(x)$ and $C_{f'}(x)$ coincide. By \cite[Theorem~3.4]{Chapman:2020a}, the conjugacy classes of both the roots and the critical points of $f(x)$ are the same as those of the roots and the critical points of $g(x)$, so it remains to apply the previous results to $g(x)$.
\end{proof}

Note that many other classical bounds on the roots of complex polynomials can be extended to the octonionic case, see, for example,~\cite{SerofioBeitesVitoria:2021}.

\begin{exmpl}
If $A$ is at least $16$-dimensional locally-complex Cayley--Dickson algebra, then Corollary~\ref{cor:nonmonic-polynomial} does not hold for $f(x) \in A[x]$. Indeed, consider $a = e_1 + e_{10}$ and $b = e_7 + e_{12}$ from Example~\ref{exmpl:companion-fails}, and let $f(x) = ax$. Then any element from $\mathbb{R}b$ is a root of $f(x)$, so there is no sphere which contains all roots of $f(x)$.
\end{exmpl}

The Gauss--Lucas theorem for $\mathbb{C}$ states also that if $K(f)$ is not a line segment then $f'(z) = 0$ and $f(z) \neq 0$ imply that $z$ is an interior point of $K(f)$. However, as the following example shows, this is not the case for a locally-complex Cayley--Dickson algebra $A$ of dimension at least $4$.

\begin{exmpl}
	Let $E_n = \{ e_m \; | \; m = 0, \dots, 2^n - 1\}$ be the standard basis of $A$, $i = e_1$, $j = e_2$. Consider $f(x) = (x^2 - 1)(x - i)^2 + j \in A[x]$. Then $f_i(x) = (x^2 - 1)(x - i)^2$ and $f'(x) = f'_i(x)$, so $f'(i) = f_i(i) = 0$ but $f(i) \neq 0$. Clearly, $i$ belongs to the boundary of $K_{\mathbb{C}_i}(f_i)$, and thus it belongs to the boundary of $\sn(f)$. The disadvantage of this example is that for $I \perp 1, i, j$ we have $f_I(x) = (x^2 - 1)^2$, so $K_{\mathbb{C}_I}(f_I)$ is just a line segment. However, if $I \perp 1$ and $I \not\perp i,j$ (such values of $I$ are dense in the imaginary part of the unit sphere of $A$), then $f_I(x) = x^4 - \alpha x^3 - 2 x^2 + \alpha x + 1 + \beta$, where $\alpha, \beta \in \mathbb{R}I \setminus \{ 0 \}$. One can show that in this case $K_{\mathbb{C}_I}(f_I)$ has dimension $2$.
\end{exmpl}

\begin{rem}
In this context it is worth to mention another classical result, stating that a non-constant complex polynomial $f(x)$ is divisible by its derivative $f'(x)$ if and only if $f(x)=c(x-a)^n$ for some $c,a\in \mathbb{C}$ and $n\in \mathbb{N}$. This can be rephrased in the following way: every root of $f'(x)$ is a root of $f(x)$ if and only if $f(x)=c(x-a)^n$. This version does not extend to $\mathbb{H}$, and thus not to $\mathbb{O}$ either. Indeed, consider $f(x)=\frac{1}{3} x^3-\frac{i+j}{2} x^2+i j x-\frac{1}{6} j+\frac{1}{2}i$, then the only root $j$ of $f'(x)=x^2-(i+j)x+ij$ is also a root of $f(x)$, despite $f(x)$ not being of the desired form. The issue of divisibility requires further investigation.
\end{rem}

\section{Jensen's Theorem for Cayley--Dickson polynomials}\label{sec:jensen}

Jensen's theorem is a classical improvement of the bound given by the Gauss--Lucas theorem for the special case of polynomials with real coefficients. In this section we extend this theorem to locally-complex Cayley--Dickson algebras. First we recall the original theorem.

Given $z\in\mathbb{C}$, denote as usual by $\bar{z}$ its complex-conjugate. The circle with the center $\Re(z)$ and radius $|\Im(z)|$ is called a Jensen circle. In particular, its diameter is $2|\Im(z)|$ and it passes through $z$ and $\bar{z}$. Jensen's theorem states that the non-real critical points of a polynomial $f(x)$ with real coefficients lie on or within the Jensen circles defined by the complex-conjugate pairs of non-real roots of $f(x)$.

First, we generalize Jensen circles to locally-complex Cayley--Dickson algebras. Let $A$ be a locally-complex Cayley--Dickson algebra. For $\lambda\in{A}$ we define the Jensen sphere of $\lambda$ to be the ($2^n-1$)-dimensional sphere with
radius $|\Im(\lambda)|$ and center $\Re(\lambda)$. We remark that, by Proposition~\ref{prop:quadratic-equivalence-LCCDA}, all the points in the quadratic-equivalence class of $\lambda$ are on this sphere. 

\begin{prop}\label{prop:spherical-real-poly}
	Let $A$ be a locally-complex Cayley--Dickson algebra, and $f(x)\in A[x]$ be a polynomial with real coefficients, then all non-real roots of $f(x)$ are spherical. 
\end{prop}

\begin{proof}
	Consider the proof of Lemma~\ref{lem:if-spherical}. We performed a division with remainder of $f(x)$ by $x^2-Tx+N$, where $T=\tr(\lambda)$ and $N=\norm(\lambda)$:
	\[f(x) = g(x)(x^2-Tx+N) + ax+b.\]
	In this special case, we have $a$ and $b$ real. 
	Any root $r$ of $x^2-Tx+N$ must then satisfy $ar+b=0$. 
	We thus have two cases: if $a$ and $b$ are both $0$,  then Theorem~\ref{thm:decompose-iff-spherical} proves that the root is spherical. Otherwise, $ar+b=0$ has a unique solution which must be real.
	
\end{proof}

\begin{thm} \label{thm:jensen}
	Let $A$ be a locally-complex Cayley--Dickson algebra, and $f\in A[x]$ be a polynomial with real coefficients, then the spherical critical points of $f(x)$ lie on or within the Jensen spheres defined by the spherical roots of $f(x)$. 
\end{thm}

\begin{proof}
Let $\lambda$ be a spherical critical point of $f(x)$. This point generates a quadratic extension $C$ of $\mathbb{R}$, isomorphic to $\mathbb{C}$. By Jensen's theorem, applied to $C$, the critical point $\lambda$ is on or contained within a Jensen circle defined by a pair of complex-conjugate roots $z,\bar{z}\in{C}$ of $f(x)$. The root $z$ is spherical by Proposition~\ref{prop:spherical-real-poly}, and so the Jensen circle of $z$ is contained in the Jensen sphere at $z$, since they are both centered at $\Re(z)$ and have the same radius (by Proposition~\ref{prop:quadratic-equivalence-LCCDA}). Thus $\lambda$ must be on or within the Jensen sphere. 
\end{proof}

\begin{rem}
We finish by noticing that Theorem~\ref{thm:jensen} provides an improvement to the bounds given Theorem~\ref{thm:gauss--lucas}, as the companion polynomial has real coefficients, and so the spherical critical points of $f(x)$ are inscribed by the (finitely many) Jensen spheres of the spherical roots of $C_f(x)$.
\end{rem}

\section*{Acknowledgements}

The first two authors thank the organizers of the CIMPA 2020 school on Non-associative Algebras and their Applications (taking place in August 2021), which triggered this collaboration. The work of the second and the fourth authors was partially financially supported by the grant RSF 21-11-00283.

\bibliographystyle{abbrv}
\bibliography{bibfile}

\begin{thebibliography}{10}

\bibitem{biss2008large}
D.~K. Biss, D.~Dugger, and D.~C. Isaksen.
\newblock Large annihilators in {C}ayley--{D}ickson algebras.
\newblock {\em Communications in Algebra}, 36(2):632--664, 2008.

\bibitem{bremner2001identities}
M.~Bremner and I.~Hentzel.
\newblock Identities for algebras obtained from the {C}ayley--{D}ickson
  process.
\newblock {\em Communications in algebra}, 29(8):3523--3534, 2001.

\bibitem{brevsar2011locally}
M.~Bre{\v{s}}ar, P.~{\v{S}}emrl, and {\v{S}}.~{\v{S}}penko.
\newblock On locally complex algebras and low-dimensional {C}ayley--{D}ickson
  algebras.
\newblock {\em Journal of Algebra}, 327(1):107--125, 2011.

\bibitem{chan2006conjugacy}
K.-C. Chan and D.~{\v{Z}}. {\DJ}okovi{\'c}.
\newblock Conjugacy classes of subalgebras of the real sedenions.
\newblock {\em Canadian Mathematical Bulletin}, 49(4):492--507, 2006.

\bibitem{Chapman:2020b}
A.~{Chapman}.
\newblock {Factoring octonion polynomials}.
\newblock {\em {Int. J. Algebra Comput.}}, 30(7):1457--1463, 2020.

\bibitem{Chapman:2020a}
A.~{Chapman}.
\newblock {Polynomial equations over octonion algebras}.
\newblock {\em {J. Algebra Appl.}}, 19(6):10, 2020.
\newblock Id/No 2050102.

\bibitem{ChapmanMachen:2017}
A.~Chapman and C.~Machen.
\newblock Standard polynomial equations over division algebras.
\newblock {\em Adv. Appl. Clifford Algebr.}, 27(2):1065--1072, 2017.

\bibitem{ChapmanVishkautsan:2021}
A.~Chapman and S.~Vishkautsan.
\newblock Fixed points of polynomials over division rings.
\newblock {\em Bulletin of the Australian Mathematical Society}.
\newblock to appear.

\bibitem{EakinSathaye:1990}
P.~Eakin and A.~Sathaye.
\newblock On automorphisms and derivations of {C}ayley-{D}ickson algebras.
\newblock {\em J. Algebra}, 129(2):263--278, 1990.

\bibitem{furey2018three}
C.~Furey.
\newblock Three generations, two unbroken gauge symmetries, and one
  eight-dimensional algebra.
\newblock {\em Physics Letters B}, 785:84--89, 2018.

\bibitem{Ghiloni:2018}
R.~Ghiloni and A.~Perotti.
\newblock The quaternionic {G}auss--{L}ucas theorem.
\newblock {\em Ann. Mat. Pura Appl. (4)}, 197(6):1679--1686, 2018.

\bibitem{gillard2019three}
A.~B. Gillard and N.~G. Gresnigt.
\newblock Three fermion generations with two unbroken gauge symmetries from the
  complex sedenions.
\newblock {\em The European Physical Journal C}, 79(5):1--11, 2019.

\bibitem{Guterman-Zhilina:2020}
A.~E. Guterman and S.~A. Zhilina.
\newblock {C}ayley--{D}ickson split-algebras: doubly alternative zero divisors
  and relation graphs.
\newblock {\em Fundam. Prikl. Mat.}, 23(3):95--129, 2020.
\newblock [In Russian, translation is to appear in {\em J. Math. Sci. (New
  York)}].

\bibitem{kuwata2004born}
S.~Kuwata.
\newblock Born--{I}nfeld {L}agrangian using {C}ayley--{D}ickson algebras.
\newblock {\em International Journal of Modern Physics A}, 19(10):1525--1548,
  2004.

\bibitem{Marden:1966}
M.~Marden.
\newblock {\em Geometry of polynomials}.
\newblock Mathematical Surveys, No. 3. American Mathematical Society,
  Providence, R.I., second edition, 1966.

\bibitem{McCrimmon:2004}
K.~McCrimmon.
\newblock {\em A taste of {J}ordan algebras}.
\newblock Universitext. Springer-Verlag, New York, 2004.

\bibitem{Moreno:1998}
G.~Moreno.
\newblock The zero divisors of the {C}ayley--{D}ickson algebras over the real
  numbers.
\newblock {\em Bol. Soc. Mat. Mexicana (3)}, 4(1):13--28, 1998.

\bibitem{moreno2006alternative}
G.~Moreno.
\newblock Alternative elements in the {C}ayley--{D}ickson algebras.
\newblock {\em Topics in Mathematical Physics, General Relativity and Cosmology
  in Honor of Jerzy Pleba{\~n}ski, World Sci. Publ., Hackensack, New Jersey},
  pages 333--346, 2006.

\bibitem{RahmanSchmeisser:2002}
Q.~I. Rahman and G.~Schmeisser.
\newblock {\em Analytic theory of polynomials}, volume~26 of {\em London
  Mathematical Society Monographs. New Series}.
\newblock The Clarendon Press, Oxford University Press, Oxford, 2002.

\bibitem{Schafer:1954}
R.~D. Schafer.
\newblock On the algebras formed by the {C}ayley--{D}ickson process.
\newblock {\em Amer. J. Math.}, 76:435--446, 1954.

\bibitem{SerofioBeitesVitoria:2021}
R.~Ser\^{o}dio, P.~D. Beites, and J.~Vit\'{o}ria.
\newblock Bounds for the zeros of unilateral octonionic polynomials.
\newblock {\em An. \c{S}tiin\c{t}. Univ. ``Ovidius'' Constan\c{t}a Ser. Mat.},
  29(3):243--267, 2021.

\bibitem{zhang:1997}
F.~Zhang.
\newblock Quaternions and matrices of quaternions.
\newblock {\em Linear Algebra Appl.}, 251:21--57, 1997.

\bibitem{SSSZ}
K.~A. Zhevlakov, A.~M. Slin'ko, I.~P. Shestakov, and A.~I. Shirshov.
\newblock {\em Rings that are nearly associative}, volume 104 of {\em Pure and
  Applied Mathematics}.
\newblock Academic Press, Inc. [Harcourt Brace Jovanovich, Publishers], New
  York-London, 1982.
\newblock Translated from the Russian by Harry F. Smith.

\end{thebibliography}

\end{document}